\documentclass[10pt]{amsart}

\usepackage[a4paper]{geometry}

\usepackage{graphicx}
\usepackage{amsmath,amsthm, amsfonts, amssymb}

\newtheorem*{theorem*}{Theorem}
\newtheorem{lemma}{Lemma}[section]
\newtheorem{proposition}[lemma]{Proposition}

\theoremstyle{definition}
\newtheorem{definition}[lemma]{Definition}
\theoremstyle{remark}

\newcommand{\real} {\mathbb{R}}

\newcommand{\R}{{\bf\sf R}}
\newcommand{\subR}{{\bf\sf Q}}

\newcommand{\K}{{\bf\sf K}}

\newcommand{\dist}{\mbox{dist}}
\newcommand{\diam}{\mbox{diam}}

\title[Approximation of measures with non-zero Lyapunov exponents]{Uniform hyperbolic approximations \\ of measures with non-zero Lyapunov exponents}
\author{Stefano Luzzatto} 
\address{
Abdus Salam International Centre for Theoretical Physics, Strada Costiera 11, 34151 Trieste, Italy
}
\email{luzzatto@ictp.it}

\author{ 
Fernando J S\'anchez-Salas}
\address{
Departamento de Matem\'aticas, Facultad Experimental de Ciencias, Universidad del Zulia, Avenida Universidad, Edificio Grano de Oro, Maracaibo, Venezuela
}
\email{fjss@fec.luz.edu.ve}

\date{24 November 2011}



\thanks{Most of this work was carried out at the Abdus Salam International Centre for Theoretical Physics (ICTP). The second named author was partially supported by the Associateship Programme of ICTP}

\begin{document}
\begin{abstract}
{\em We show that for any \( C^{1+\alpha} \) diffeomorphism of a compact Riemannian manifold, every non-atomic, ergodic, invariant probability measure with nonzero Lyapunov exponents is approximated by uniformly hyperbolic sets in the sense that 
there exists a sequence \( \Omega_{n} \) of compact, topologically transitive, locally maximal, uniformly hyperbolic sets, such that  for {any} sequence \( \{\mu_{n}\} \) 
of \( f \)-invariant ergodic probability measures with \( supp (\mu_{n}) \subseteq \Omega_{n} \) we have \( \mu_{n}\to \mu \) in the weak-\( * \) topology. 
}
\end{abstract}

\maketitle

\section{Introduction}

Let \( M \) be a compact Riemannian manifold,  let \( f: M \to M  \) be a \( C^{1+\alpha} \) diffeomorphism and let \( \mu \) be an \( f \)-invariant ergodic probability measure.  Then, by Oseledec's Theorem \cite[Theorem S.2.9]{katok.mendoza}, the Lyapunov exponents 
\begin{equation}\label{lyapunov.exponents}
\chi(x,v):=\lim_{n\to \pm\infty} \frac 1n \log \|Df^{n}_{x}(v)\|. 
\end{equation}
are well defined for \( \mu \) almost every \( x\in M \) and for every non-zero vector \( v\in T_{x}M \). 

\begin{definition}
\( \mu \) is \emph{hyperbolic} if \( \chi(x,v)\neq 0 \) for \( \mu \) a.e.   \( x\in M \) and every \( v\in T_{x}M\setminus \{0\} \).
\end{definition}

This notion of hyperbolicity was introduced by Pesin \cite{pesin}, see also \cite{HasKat95, pesin.barreira},  as a far-reaching generalization of the well-known notion of \emph{uniform hyperbolicity} \cite{HasKat95, Has02, New80, Rob96}; it implies that \( \mu \) almost every point has a decomposition of the 
tangent space into subspaces on which vectors exhibit exponential contraction or expansion under iteration by the derivative map \( Df \). Crucially however, and in contrast with the uniformly hyperbolic situation, 
this decomposition is only measurable as opposed to continuous, and the expanding or contracting behaviour is only asymptotic, so that vectors which eventually exhibit exponential growth may suffer unbounded contraction in 
finite time. For this  reason  it is often referred to as \emph{nonuniform hyperbolicity}. 
It is significantly more general than uniform hyperbolicity and there are no topological obstructions to its occurrence:  any compact 
smooth manifold of dimension $\geq 2$ admits, for example, a volume-preserving $C^{\infty}$ diffeomorphism for which every invariant measure is hyperbolic \cite{DolPes02}. 
We refer to Section \ref{sec:pesinth} for additional background on nonuniform hyperbolicity, 
and to \cite{HasKat95, Has02, pesin.barreira} for a thorough exposition of the theory.

In this paper we address the question of whether nonuniformly hyperbolic dynamics can be \emph{approximated} by uniformly hyperbolic dynamics. 

\begin{theorem*}\label{main}
Let \( M \) be a  compact Riemannian manifold,  $f : M \rightarrow M$  a $C^{1+\alpha}$ diffeomorphism, \( \mu \) a non-atomic, ergodic, \( f \)-invariant, hyperbolic, Borel probability measure $\mu$.  Then there exists a sequence \( \{\Omega_{n}\} \) of topologically transitive, locally maximal, uniformly hyperbolic compact \( f \)-invariant sets such that \( \mu_{n}\to \mu \) in the weak-\( * \) topology for {\em any} sequence \( \{\mu_{n}\} \) 
of \( f \)-invariant ergodic probability measures with \( supp (\mu_{n}) \subseteq \Omega_{n} \). 
\end{theorem*}

This  generalizes various well known results. In particular we mention the pioneering result of 
\cite{katok} where it is proved that hyperbolic measures can be  approximated by  measures supported on (uniformly) hyperbolic periodic orbits. This has spurred a number of results, many of them quite recent,  showing that various dynamical quantities can be approximated by the corresponding quantities on hyperbolic periodic orbits, 
see for example \cite{Dai, Hirayama, LiangLiuSun, Ugarcovici, WangSun, WolfGelfert}. In the special case where \( \mu \) is a  Sinai-Ruelle-Bowen measure  (i.e. a measure with a particular absolute continuity property on unstable manifolds, see \cite{pesin.barreira} for precise definitions), it was proved in \cite{sanchez.salas}, generalizing earlier work \cite{mendoza} for two dimensional systems, that \( \mu \) can be approximated by certain particular ergodic 
measures $\mu_n$ supported on horseshoes $\Omega_n$ with arbitrarily large unstable Hausdorff dimension (that is,  the Hausdorff dimension of unstable Cantor sets $\Omega_n \cap W^u(x)$).
It was also proved in \cite{katok.mendoza} that a non-atomic hyperbolic measure $\mu$ can be approximated by ergodic measures $\mu_n$ supported 
on horseshoes $\Omega_n$ such that $h({\mu_n}) \to h({\mu})$, where $h(\mu)$ is the metric entropy.  In both of these cases, the approximating measures are ``maximizing measures'' of the Hausdorff dimension and of the entropy respectively, and the methods used to obtain the results involve some variational principles. More recently, \cite{gelfert} has applied analogous methods in the setting of 
non uniformly expanding (non invertible)  endomorphisms and considered 
the approximation of the topological pressure. 

In this note we address the approximation problem from a more topological point of view, proving that uniformly hyperbolic subsets $\Omega_{n}$ can be  chosen such that \emph{all its ergodic measures} are uniformly close to $\mu$. We emphasize that this result is not contained in any of the existing papers in the literature. Moreover, although we do use some basic facts from the theory of nonuniformly hyperbolic systems, the rest of the arguments and methods are very simple and natural. 
We divide the proof into a few simple steps. In Section \ref{sec:pesinth} we give some very basic background on Pesin Theory and state a standard Proposition which we will use in the construction.  In Section \ref{sec:geo} we introduce a geometric model of what we call \emph{variable-time horseshoes}. This generalizes the standard horseshoe construction in the case where the various branches have variable return time and constitutes the basis of our construction of the approximating uniformly hyperbolic sets \( \Omega_{n} \).  In Section  
\ref{exists.qg.branches} we show that variable time horseshoes actually exist and that we can choose them satisfying some ``quasi-genericity'' properties to be defined below. Finally, in Section 
\ref{qg.horseshoes} we show that these quasi-generic horseshoes actually approximate the measure \( \mu \) in the desired manner. 

\section{Background on Pesin Theory}
\label{sec:pesinth}

We recall in this section some well known results from the theory of systems with non-zero Lyapunov exponents and refer the to reader to 
\cite{katok.mendoza, pesin.barreira} for the details.
In particular we state a simple and 
relatively standard result which we will use in the subsequent sections. We emphasize that, apart from this result, the rest of the construction and arguments in the paper are completely self contained. For the convenience of readers who are already familiar with this background material we state the Proposition precisely here and postpone the definitions to the subsequent subsections. 

\begin{proposition}\label{pseudo.markov.property}
Let $\mu$ be a non-atomic, hyperbolic $f$-invariant Borel probability. Then, for every small \( \delta > 0 \) 
there exists a ``rectangle'' $\R$ and a subset of positive measure $\Lambda \subset \R$ such that for every \(  z \in \Lambda \) and \( m\geq 1 \) with 
\( f^{m}(z) \in \Lambda \) there exist a ``quasi-horizontal full width cylinder'' $S_z \subset \R$ containing $z$, a ``quasi-vertical full height 
cylinder'' $U_z \subset \R$ containing $f^m(z)$ and a hyperbolic diffeomorphism $f^m : S_z \to U_z$ such that
$$
\diam(f^i(S)) \leq \delta \quad\text{for all}\quad i = 0 , \cdots , m.
$$
\end{proposition}

In the rest of this Section, we give the precise definitions and constructions required to understand the statement of Proposition \ref{pseudo.markov.property} and also a sketch of the proof.

\subsection{Nonuniform hyperbolicity}
As mentioned above, for an \( f \)-invariant Borel probability measure \( \mu \), there exists a set \( \Sigma \) with \( \mu(\Sigma) =1  \) for which the Lyapunov exponents 
\( \chi(x,v) \) are well defined for every \( x\in\Sigma \) and every \( v\in T_{x}M\setminus \{0\} \) and the measure \( \mu \) is hyperbolic if all the Lyapunov exponents are non-zero. Moreover, if \( \mu \) is ergodic, as in our setting, then \( \chi(x,v) \) can only take on a finite number of values on \( \Sigma \). 
 In that case, there exists  a constant $\chi $ satisfying
 \begin{equation}\label{eq:chi}
  \min\{|\chi(x,v)|: x\in \Sigma, v\in T_{x}M\setminus \{0\}\}> \chi > 0. 
  \end{equation}
 Then, for all  sufficiently 
small $\epsilon > 0$ such that \( \chi-\epsilon >0 \) (or equivalently $-\chi + \epsilon < 0$),  by Oseledec's theorem there exist measurable $Df$-invariant decompositions
$$
T_xM = E^s(x) \oplus E^u(x),
$$
and tempered Borel measurable functions $C_{\epsilon}, K_{\epsilon} : \Sigma \to (0,+\infty)$ such that
$$
\begin{cases}
\|Df^n(x)v\|    \leq C_{\epsilon}(x)e^{n(-\chi+\epsilon)}\|v\|  & \ \forall \ v \in E^s(x) \ \forall \ n \geq 0 \\
\|Df^{-n}(x)v\| \leq C_{\epsilon}(x)e^{n(-\chi+\epsilon)}\|v\|  & \ \forall \ v \in E^u(x) \ \forall \ n \geq 0 \\
\end{cases}
$$
and $\angle(E^s(x),E^u(x)) \geq K_{\epsilon}(x)$, where
$$
E^s(x) := \bigoplus\limits_{\chi_i(x) < 0}E_i(x) \quad\text{and}\quad E^u(x) := \bigoplus\limits_{\chi_i(x) > 0}E_i(x).
$$
Moreover, by Tempering-Kernel Lemma \cite[Lemma S.2.12]{katok.mendoza}, we may suppose that
$$
(1+\epsilon)^{-1} \leq \dfrac{C_{\epsilon}(f(x))}{C_{\epsilon}(x)}, \dfrac{K_{\epsilon}(f(x))}{K_{\epsilon}(x)} \leq 1+\epsilon, \quad\mu-a.e.
$$
We remark that the properties given above as a consequence of the hyperbolicity of \( \mu \) can also be formulated without any reference to the measure \( \mu \) and are essentially nonuniform versions of standard uniformly hyperbolic conditions, see \cite[Theorem 6.6]{pesin.barreira}. 

\subsection{Regular neighbourhoods and  rectangles}

The fundamental starting point for understanding and working with the geometric structure of systems with non-zero Lyapunov exponents is the notion of 
{\em regular neighbourhood}.
Our first step is to introduce a linear coordinate system $L(x): \real^n \to T_xM$, for every $x \in \Sigma$, such that $A(x)$, the representative of the derivative $Df(x)$ in the new coordinates, is a diagonal block matrix adapted to Oseledec's 
decomposition. The map $x \in \Sigma \mapsto L(x) \in GL_n(\real)$ is Borel measurable
and the coordinate changes are {\em tempered}:
$$
\limsup_{n \to \pm\infty}\dfrac{1}{n}\max\{\log\|L(f^n(x))\|, \log\|L^{-1}(f^n(x))\|\} = 0 \quad\mu\text{-a.e.}
$$
Let $<\cdot\,,\,\cdot>$ denote the standard inner product in $\real^n$, then a new measurable {\em Lyapunov metric} $<\cdot\,,\,\cdot>'_x$ is defined 
so that $L(x) : (\real^n, <\cdot\,,\,\cdot>) \to (T_xM,<\cdot\,,\,\cdot>'_x)$ is a linear isometry and such that
$$
e^{-\chi - \epsilon} \leq \|A^u(x)^{-1}\|, \|A^s(x)\| \leq e^{-\chi + \epsilon} \quad\mu\text{-a.e.},
$$
where $A^s$ and $A^u(x)$ are restrictions of $A(x)$ to the stable and unstable subspaces, respectively.
The new norm is equivalent to the Riemannian metric $g_x$, bounded by a measurable tempered correction $D(x)$ depending only on the Riemannian structure of $M$ and $\angle(E^s(x),E^u(x)>$: 
$$
\dfrac{\|v\|_x}{\sqrt{2}} \leq \|v\|'_x \leq D(x)\|v\|_x \quad\text{for every}\quad v \in T_xM-\{0\}.
$$
This process is known as {\em $\epsilon$-reduction}, see \cite[Theorem S.2.10]{katok.mendoza} and \cite[\S 5.5]{pesin.barreira}.

Our second step is to introduce coordinate systems in which the dynamics is essentially uniformly hyperbolic. The domain in which these local coordinate systems apply are called regular neighbourhoods. 
For this we define {\em Lyapunov charts}
$$
\psi_x := \exp_x \circ L(x)
$$
where $\exp_x : B(0,r_M) \subset T_xM \to M$ are local geodesic coordinates and $r_M > 0$ the injectivity radius of $(M,g)$. It is proved in \cite[Theorem S.3.1]{katok.mendoza} 
and \cite[\S 8.7]{pesin.barreira} that there exists a tempered Borel measurable function $r: \Sigma_0 \to (0,+\infty)$ such that
$$
\psi_x : B(0,r(x)) \subset \real^n \to M \quad\text{with}\quad \psi_x(0) = x,
$$
is an embedding; moreover
\begin{equation}\label{perturbative.argument}
\dist_{C^1}(f_x \mid_{B(0,r(x))},A(x)) < \epsilon, 
\end{equation}
where $f_x := \psi_{f(x)}^{-1}\circ{f}\circ\psi_x$ is the representative of $f$ in the given coordinates. In particular, 
$$
f_x(v,w) = (A^s(x)v + \phi^s(v,w),A^uw + \phi^s(v,w)) \quad\text{with}\quad \|(\phi^s,\phi^u)\|_{C^1(B(0,r(x))} < \epsilon.
$$
Let $\sigma_{x} : [-1,1]^n \to [-t(x),t(x)]^n$ be the linear rescaling onto the maximal cube contained in $B(0,r(x))$.
\begin{definition}
The rectangle $\R(x)$ is the image of the cube $[-1,1]^n \subset \real^n$ under
$$
e_x := \psi_x \circ \sigma_{x} : [-1,1]^n \to M.
$$
\end{definition}
\subsection{Cylinders and hyperbolic branches}

The crucial feature of a regular neighbourhood is that it admits a coordinate system in which the dynamics is essentially uniformly hyperbolic and in 
particular it defines locally certain approximate \emph{stable} and \emph{unstable} directions which are transversal to each other. 
Let us fix some $0 < \gamma < 1/2$ and decompose the unit radius cube as a product $I^n = I^s \times I^u$.

\begin{definition}
A {\em stable admissible manifold} is a graph $\gamma^s = \{e_x(z,\hat{s}(z)): z \in I^s\}$, 
where $\hat{s} \in C^1(I^s,I^u)$ is a smooth map with $Lip(\hat{s}) := \sup_{z \in I^s}\|D\hat{s}(z)\| \leq \gamma$.
\end{definition}

Admissible manifolds endow $\R(x)$ with a product structure: any given pair of admissible manifolds $\gamma^s$ and $\gamma^u$ intersectes transversally at 
a unique point with an angle bounded from below. Moreover, the map $(\gamma^s, \gamma^u) \to \gamma^s\cap\gamma^u$ so defined satisfies a Lipschitz 
condition  \cite[\S 3.b]{katok.mendoza} and \cite [\S 8]{pesin.barreira}.
The transversal structure of the admissible stable and unstable manifolds inside a rectangle \( \R \) allows us to define the notion of 
{\em admissible stable and unstable cylinders}. An {\em admissible stable cylinder} \( S \subseteq \R\) is a subrectangle of $\R$ whose boundaries are 
piecewise smooth sets foliated by segments of admissible stable and unstable manifolds such that the stable manifolds stretch fully across the rectangle \( \R \), and similarly, 
an {\em admissible unstable cylinder} \( U \subseteq \R\) is a subrectangle of $\R$ whose boundaries are segments of admissible stable and unstable 
manifolds such that the unstable manifolds stretch fully across the rectangle \( \R \).

The notion of admissible manifold is related to certain  cone fields $\K^s$, $\K^u$.
For every $z \in \R$ we define $\K^s_z \subset T_zM$ as the image under $De_x(p)$ evaluated at $p(z) = e^{-1}_x(z) \in I^n$, of the 
cone of width $\gamma$ 'centered' at $\real^s\oplus\{0\}$, that is, the set of vectors in $\real^n$ making an angle bounded by $\gamma$ with $\real^s\oplus\{0\}$. 
We define $\K^u_z \subset T_zM$ likewise considering a cone of width $\gamma$ 'centered' at $\{0\}\oplus\real^u$. 
Notice that admissible manifolds are exactly those smooth graph-like submanifolds whose tangent spaces rest inside stable and unstable cones.

We say that a $C^1$ diffeomorphism $g : S \to U$ between admissible cylinders is {\em hyperbolic} if it preserves the cone fields $\K^s$ and $\K^u$, that is,
$$
Dg(z)\K^u_z \subset int\,\K^u_{g(z)} \ \forall \ z \in S\quad\text{and}\quad Dg^{-1}(z)\K^s_z \subset int\,\K^s_{g^{-1}(z)}  \ \forall \ z \in U,
$$

\begin{definition}
Let $\R$ and $\subR$ be regular rectangles. If some iterate \( f^{m} \) maps an admissible stable cylinder  $S \subset \R$ diffeomorphically and hyperbolically 
to an admissible unstable cylinder $U \subset \subR$, we shall say that
$$
f^m : S \to U
$$
is a {\em hyperbolic branch}.
\end{definition}

Proposition \ref{pseudo.markov.property}, which we stated above and will use in an essential way below,  is a result about the existence of hyperbolic branches. 

\subsection{Uniformly hyperbolic Pesin sets}
We now introduce a standard ``filtration'' of \( \mu \) almost every point which gives a countable number of nested, 
uniformly hyperbolic (but not \( f \)-invariant) sets, often referred to as ``Pesin sets'', whose points admit uniform hyperbolic bounds and uniform 
lower bounds on the sizes of the local stable and unstable manifolds.

For  $\chi > 0$ as in \eqref{eq:chi} above, and every positive integer $\ell > 0$, we define a (possibly empty) 
compact (not necessarily invariant) set $\Lambda_{\chi,\ell} \subset M$ such that $E^s|_{\Lambda_{\chi,\ell}}$ and  $E^u |_{\Lambda_{\chi,\ell}}$ vary continuously with the point $x \in \Lambda_{\chi,\ell}$ and  such that 
$$
\begin{cases}
\|Df^n(x)v\| \leq {\ell}e^{-n\chi}\|v\| & \|Df^{-n}(x)v\| \geq \ell^{-1}e^{n\chi}\|v\| \ \forall \ v \in E^s(x) \ \forall \ n \geq 0 \\
\|Df^{-n}(x)v\| \leq {\ell}e^{-n\chi}\|v\| & \|Df^n(x)v\| \geq \ell^{-1}e^{n\chi}\|v\| \ \forall \ v \in E^s(x) \ \forall \ n \geq 0.
\end{cases}
$$
Moreover,  the angles between the stable and unstable subspaces satisfy
$$
\angle(E^s(x),E^u(x)) \geq \ell^{-1}
$$
for every  $x \in \Lambda_{\chi,\ell}$. As the rate of hyperbolicity of $\mu$ is bounded from below by $\chi > 0$ 
we have 
$$
\mu(\Lambda_{\chi,\ell}) \to 1 \quad\text{as}\quad \ell \to +\infty.
$$
The following result is proved in  \cite[Theorem S.4.3]{katok.mendoza}.
\begin{lemma}\label{alpha}
For every $0 < \alpha < 1$ there exists a $\Lambda = \Lambda_{\chi,\ell}$ with $\mu(\Lambda) \geq 1-\alpha$ 
such that $\R(x)$ vary continuously with $x \in \Lambda$, meaning that linear distortion $D(x)$, size of Lyapunov charts $r(x)$ 
and $x \mapsto e_x \in Embedd(I^s\times{I^u}, M)$ are continuous functions on $\Lambda$.
\end{lemma}

We therefore fix some \( \alpha \) and let \( \Lambda \) be the set given by Lemma \ref{alpha}. 
Let $r_{\Lambda} := \min_{x \in \Lambda}r(x) > 0$ be the minimal radius for Lyapunov charts for $x \in \Lambda$ and denote 
$$
\sigma^0_x : [-1,1]^n \to [-t_0,t_0]^n
$$
the linear rescaling onto the maximal cube contained in $B(0,r_{\Lambda}/2)$.

\begin{definition}
Given $0 < h < 1$, we define the {\em $h$-reduced rectangle of center $x$} as the image of the cube $[-h,h]^n \subset [-1,1]^n$ under the 
parametrization $e^0_x := \psi_x \circ \sigma^0_{x}$, i.e.
$$
R^0(x,h) := e^0_x([-h,h]^n).
$$
\end{definition}

The following statement is contained in \cite[Theorem S.4.16]{katok.mendoza}.  

\begin{proposition}\label{hyperbolic.return.lemma}
There exists a constant $C=C(\Lambda) > 0$  depending only on the Pesin set $\Lambda$, such that for any $0 < h < 1$ there exists 
$\beta = \beta(h,\Lambda) > 0$ with the following property: for every $x,y,z \in \Lambda$ with 
$f^m(z) \in \Lambda$ and $\dist(x,z),\dist(y,f^m(z)) < \beta$, there exists an admissible stable cylinder $S_z \subset \R^0(x,h)$ containing $z$ 
and an admissible unstable cylinder $U_{z} \subset \R^0(y,h)$ contaning $f^m(z)$ such that $f^m : S_z \to U_z$
is a hyperbolic branch and $\diam(f^j(S_z)) \leq Ch$, for every $j = 0, \cdots , m-1$. 
\end{proposition}

\begin{proof}
[Proof of Proposition \ref{pseudo.markov.property}]
The statement is now a straightforward corollary of Proposition \ref{hyperbolic.return.lemma}. 
Let $\delta > 0$ be given and fix a hyperbolic Pesin set $\Lambda$, a point $x \in \Lambda\cap{supp \ {\mu}}$. We choose 
$0 < h_{\beta} < h_{\delta} < 1$ such that 
$$
C(\Lambda)h_{\delta} < \delta \quad\text{and}\quad \diam(\R^0(x,h_{\beta})) < \beta.
$$
Then, if we let $\R := \R^0(x,h_{\delta})$, $\subR := \R^0(x,h_{\beta})$ and $\Lambda_0 := \subR \cap \Lambda$, then Proposition \ref{hyperbolic.return.lemma}
shows that every return from $z \in \Lambda_0$ gives rise to  a hyperbolic branch $f^m : S \to U$ with $\diam(f^j(S)) < \delta$ for $j = 0, \cdots, m-1$ as 
claimed.
\end{proof}

\section{Horseshoes with variable return time}
\label{sec:geo}

In this Section we introduce the notion, and a geometric model of, \emph{variable-time horseshoes}. Combining this construction with the statement in Proposition \ref{pseudo.markov.property} and the results of Section \ref{exists.qg.branches} will then yield the statement in our Theorem. 
Our geometric model is  defined by a finite collection $\mathcal S$ of pairwise disjoint {\em stable cylinders} 
$\{S_{1},.., S_{N}\}$ and corresponding pairwise disjoint collection $\mathcal U$ of 
{\em unstable cylinders} \ $\{U_{1},..., U_{N} \}$ contained in a rectangle $\R$. 
We  assume that  for each \( i = 1,...,N \) there exists a time \( m_{i} \) such that 
$
f^{m_{i}}: S_{i}\to U_{i}
$
is a uniformly hyperbolic diffeomorphism. We suppose moreover that each stable cylinder \( S_{i} \) ``crosses'' all \( U_{i} \)'s transversally and each \( U_{i} \) ``crosses'' 
all \( S_{i} \)'s transversally as in the  standard set up for horseshoes (see e.g. \cite{wiggins}). The key difference here is just that we are allowing the times \( m_{i} \) to depend on the cylinder. This will play an important role in our construction. 

To construct the horseshoe we  define  piecewise smooth invertible maps \( F: \mathcal S \to \mathcal U \) and \( F^{-1}: \mathcal U \to \mathcal S \) by 
\[
F|_{S_{i}}:= f^{m_{i}}|_{S_{i}} \quad\text{ and } \quad F^{-1}|_{U_{i}}:= f^{-m_{i}}|_{U_{i}}.
\]
We then define a set \( \Omega^{*} \) as the maximal invariant set under iterations of \( F \) and \( F^{-1} \). More precisely, we define inductively decreasing sequences of families of cylinders using the following algorithm: we let $S_{i}^{(1)}:=S_{i}$, $U_{i}^{(1)}:=U_{i}$ and let 
\[
\mathcal S^{(1)}:=\bigcup_{i=1}^{N}S_{i}^{(1)}\qquad 
\mathcal U^{(1)}:= \bigcup_{i=1}^{N} U_{i}^{(1)}.
\]
Proceeding inductively,  we then have  
$$
\mathcal S^{(n)}:=\bigcup_{i=1}^{N} S_{i}^{(n)}, \quad \mathcal U^{(n)}:= \bigcup_{i=1}^{N} U_{i}^{(n)}.
$$
where 
\[
S_{i}^{(n)}:= f^{-m_{i}}(U_{i}\cap \mathcal S^{(n-1)})\quad\text{ and } \quad 
U_{i}^{(n)}:= f^{m_{i}}(S_{i}\cap \mathcal U^{(n-1)}).
\]
  Each \( S_{i}^{(n)} \) is a union of \( N^{n-1} \) subcylinders contained inside the original stable set \( S_{i} \) and each  
\( U_{i}^{(n)} \) is a union of \( N^{n-1} \) subcylinders contained inside the original strip \( U_{i} \). 
It follows from the hyperbolicity of the map that the diameter of the stable subcylinders in the unstable direction is 
uniformly contracted at each iteration, and  similarly the diameter of unstable subcylinders in the stable direction
is contracted  at each iteration.
Therefore, the families $S_{i}^{(n)}$ (resp. $U_{i}^{(n)}$) are nested and their intersection converges to continuous families of admissible 
$F$-invariant laminations
\[
\mathcal F^{S}_{i}= \bigcap_{n\geq 1} S_{i}^{(n)} 
 \quad\text{and} \quad 
\mathcal F^{U}_{i}= \bigcap_{n\geq 1} U_{i}^{(n)}. 
 \]
 We then let 
 \[ 
 \mathcal F^{S}= \bigcup_{i=1}^{N}\mathcal F^{S}_{i}
 \quad\text{and}\quad
 \mathcal F^{U}= \bigcup_{i=1}^{N}\mathcal F^{U}_{i}
  \] 
Finally we define 
\[
 \Omega^* :=   \mathcal F^{S} \cap  \mathcal F^{U}. 
\]
The following statement then follows by completely standard methods in hyperbolic dynamics, see \cite{Has02} for instance.

\begin{lemma}\label{geometric.model}
$\Omega^*$ is an $F$-invariant Cantor set endowed with a hyperbolic product structure defined by the laminations of local $F$-invariant manifolds 
$\mathcal F^{S}$ and $\mathcal F^{U}$. The set 
\begin{equation}
\Omega := \bigcup_i\bigcup_{j=0}^{m_i-1}f^j(\Omega^*_i)
\end{equation}
 is  topologically transitive, locally maximal, uniformly hyperbolic, and  
$f$-invariant. 
\end{lemma}

\begin{definition}
We call \( \Omega^{*} \) a \emph{variable-time horseshoe} and \( \Omega \) the \( f \)-invariant \emph{saturate} of  \( \Omega^{*} \). 
\end{definition}

\section{Quasi-generic branches}
\label{exists.qg.branches}

Since the measure  is assumed fixed for the rest of the paper, we will omit explicit mention of \( \mu \) when there is no possibility of 
confusion. We also fix a large hyperbolic Pesin set $\Lambda$ with measure $\mu(\Lambda) > 1/2$, say.
We  fix a countable dense subset $\{\phi_i\}$ of the space $C^0(M)$ of continuous real valued functions on \( M \). 
Given two constants $\rho, s > 0$ we define the weak-\(*\) \( (\rho, s) \) neighborhood of \( \mu \)
by
\begin{equation}
{\mathcal O}(\rho, s) := \{\nu : \left|\int\phi_i{d{\mu}} - \int\phi_i{d{\nu}}\right| < \rho, \ i = 1, \cdots , s\}.
\end{equation}

\begin{definition}
We say that a point \( x \) is \( (\rho, s, n) \) \emph{quasi-generic} for the measure \( \mu \) if 
\[ 
\left|\frac 1n\sum_{j=0}^{n-1}\phi_{i}(f^{j}(x))-\int\phi_{i}d\mu\right|\leq \rho \quad \forall i\leq s.
 \] 
 A hyperbolic branch 
 \[
 f^n : S \to U
 \]
  is {\em $(\rho,s)$-quasi-generic} for \( \mu \) if every \( x\in S \) is \( (\rho, s, n )  \) quasi-generic for \( \mu \). 
\end{definition}

\begin{proposition}\label{prop:existqghorse}
For every  \( \rho, s > 0 \)  there exists a variable time horseshoe \( \Omega^{*}(\rho, s) \) defined by \( (\rho,s) \) quasi-generic branches. 
\end{proposition}

\begin{proof}
The proof  divides into two parts. First we show that for every \( \rho, s>0 \) there exist \( (\rho, s) \) quasi generic points. Then we show that 
these points can be used to construct quasi-generic branches which can be used to construct the horseshoe \( \Omega^{*} \) as in the Proposition. 

The first part  follows immediately by abstract ergodicity results. Indeed, 
there exists a set  \( \mathcal G_{\mu}\) with 
\( \mu(\mathcal G_{\mu})=1 \)
of \emph{generic} points with respect to \( \mu \) in the sense  that for every \( x\in \mathcal G_{\mu} \)
\[ 
\frac 1n\sum_{j=0}^{n-1}\phi(f^{j}(x)) \to \int \phi d\mu
 \] 
as \( n\to \infty \) for all continuous functions \( \phi \). This implies that for every \( x\in \mathcal G_{\mu} \) and every \( \rho, s > 0  \) 
there exists \( m_{0}=m_{0}(\rho, s, x) \) such that \( x \)  is \( (\rho, s, m) \) quasi-generic for every \( m \geq m_{0} \) (
the constant \( m_{0} \) is  related to the speed  the convergence of ergodic sums to the average and is in general highly non-uniform, 
in particular \( m_{0} \) can be unbounded in \( x \)). To show that these quasi-generic points can be used to construct quasi-generic branches, 
we use the hyperbolicity assumptions on the measure \( \mu \), more precisely Proposition \ref{pseudo.markov.property} on the existence of 
hyperbolic branches. It is then quite easy to see that if the point \( x \) is quasi-generic then essentially the same is true for the 
corresponding hyperbolic branch, as stated formally in the following 

\begin{lemma}\label{lem-qg}
For every \( \rho, s > 0 \), there exists \( \delta (\rho, s) >0 \) such that if \( z \in \Lambda \) is \((\rho/2, s, m) \) quasi-generic 
returning to $\Lambda$ after $m$-iterates, then the corresponding hyperbolic branch $f^m : S \rightarrow U$ is \( (
\rho, s) \) quasi-generic.
\end{lemma}

\begin{proof}
Let $\delta(\rho,s) > 0$ be a positive number such that 
\begin{equation}\label{definition.delta}
d(x,y) < \delta(\rho,s) \quad\text{implies}\quad |\phi_{i}(x)-\phi_{i}(y)| < \rho/2 \quad\text{for every} \  i \leq s.
\end{equation}
Let $z \in \Lambda$ be a point  satisfying the assumptions of the lemma and let \( S \) be the corresponding stable cylinder given by 
Proposition \ref{pseudo.markov.property}. Then for every \( x\in S \) and every $1 \leq i \leq s$ we have
\begin{eqnarray*}
\left|\frac{1}{m}\sum_{j=0}^{m-1}\phi_i(f^j(x)) - \int\phi_i{d\mu}\right| & \leq & \\
\left|\frac{1}{m}\sum_{j=0}^{m-1}\phi_i(f^j(x)) - \frac{1}{n}\sum_{j=0}^{m-1}\phi_i(f^j(z))\right| &  + & 
\left|\frac{1}{m}\sum_{j=0}^{m-1}\phi_i(f^j(z)) - \int\phi_i{d\mu}\right|.
\end{eqnarray*}

The first term can then be bounded by \( \frac{1}{m}\sum_{j=0}^{m-1} |\phi_i(f^j(x)) - \phi_i(f^j(z))|  \) 
which is \( \leq \rho/2 \) be the definition of \( \delta \) and the fact that \( diam(f^{i}(S))\leq \delta \) 
from Lemma \ref{pseudo.markov.property}. The second term is \( \leq \rho/2 \) by 
the assumption that \( z \) is \( (\rho/2, s, m) \) quasi-generic. 
\end{proof} 

Returning to the proof of Proposition \ref{prop:existqghorse} it is therefore now sufficient to show that for every \( \rho, s>0 \) there exist finitely many 
(at least two) points \( x_{1},...,x_{N} \) such that \( x_{i} \) is \( (\rho, s, m_{i}) \) quasi generic and which generate disjoint 
hyperbolic branches \( f^{m_{i}}: S_{i}\to U_{i} \) which intersect transversally and \emph{cross} each other as in the geometric model described in 
Section \ref{sec:geo} above.

First we  we let $\delta(\rho,s) > 0$ be given by Lemma \ref{lem-qg}. Then, by Proposition 
\ref{pseudo.markov.property}, there exists a small rectangle $\R$  enclosing a subset of positive measure 
$\Lambda \subset \R$ such that $\diam(f^j(S_z)) < \delta(\rho,s)$ for $j = 0, \cdots , m$ for every point $z \in \Lambda$ returning to $\Lambda$
after $m$-iterates.
Now let ${\mathcal G}_{\rho,s,n}$ denote the set of $(\rho,s,n)$ quasi generic points of $\mu$. By the ergodicity of $\mu$, for every $\rho > 0$ we have $\mu({\mathcal G}_{\rho,s,n}) \to 1$ as $n \to +\infty$ and therefore  we can choose $n > 0$ sufficiently large such that
$$
\mu(\Lambda \cap {\mathcal G}_{\rho,s,n}) > 0.
$$  
We shall prove that there exists at least two $(\rho,s)$ quasi generic branches $f^{m_i}: S_i \to U_i$, $i = 1,2$ 
associated to suitable returns to $\Lambda \cap {\mathcal G}_{\rho,s,n}$ with $m_i \geq n$, and such that $S_1 \cap S_2 = \emptyset$ and $U_1 \cap U_2 = \emptyset$. This is essentially obvious since points return infinitely often and each return gives rise to a hyperbolic branch with stable and unstable cylinders of exponentially small diameter in the unstable and stable directions respectively. For completeness we give a more formal sketch of this argument. 

We start choosing a point $z_1 \in \Lambda\,\cap\,{\mathcal G}_{\rho,s,n}$ and $m_ 1 \geq n$ such that $f^{m_1}(z_1) \in \Lambda$. 
By Proposition \ref{pseudo.markov.property} and Lemma \ref{lem-qg}, this gives rise to  a $(\rho,s)$ quasi generic hyperbolic branch $f^{m_1} : S_1 \to U_1$ in $\R$.
As $\mu$ is non-atomic, we can choose a small region  ${\mathcal N} \subset \R$ at a strictly positive distance from  $S_1$ and  $U_1$,  and such that 
$\mu({\mathcal N} \cap \Lambda) > 0$. 
By Poincar\'e recurrence, there exists a $(\rho,s,n)$ quasi generic point $z_2 \in {\mathcal N} \cap \Lambda$ returning infinitely often to 
${\mathcal N} \cap \Lambda$ and, by Pesin's theory, a local invariant manifold $W^s(z_2,\R) \subset \R$ which is admissible in $\R$ and such that 
\( W^s(z_2,\R) \cap S_1 = \emptyset \) by our choice of \( \mathcal N \). 
Choosing a return time $m_2 \geq n$ such that $f^{m_2}(z_2) \in {\mathcal N}$ we obtain, also by Proposition \ref{pseudo.markov.property} and Lemma \ref{lem-qg}, a quasi generic branch 
$f^{m_2} : S_2 \to U_2$ such that $W^s(z_2,\R) \subset S_2$.
If $S_1 \cap S_2 = \emptyset$ and  
$U_1 \cap U_2 = \emptyset$
then we are done. Otherwise we iterate forward and backward the branch $f^{m_2}: S_2 \to U_2$  to get  
strictly narrower stable and unstable cylinders disjoint from $S_1$ and \( U_{1} \), thus forming an independent quasi generic branch with larger return time.
This completes the proof of Proposition \ref{prop:existqghorse}.
\end{proof}

\section{Quasi-generic horseshoes}
\label{qg.horseshoes}

\begin{proposition}\label{prop:qghorse}
Let \( \rho, s > 0 \) and suppose there exists a variable time horseshoe \( \Omega^{*}(\rho, s) \) defined by \( (\rho,s) \) quasi-generic branches. 
Then every  \( f \)-invariant ergodic probability measure $\mu_{\Omega}$ supported on $\Omega(\rho,s)$, 
the $f$ invariant saturate of $\Omega^{*}(\rho, s)$, satisfies  $\mu_{\Omega} \in {\mathcal O}( 3 \rho, s)$.
\end{proposition}

Proposition \ref{prop:qghorse} together with Proposition \ref{prop:existqghorse}  clearly imply our Theorem. Indeed,  choosing sequences $\rho_n \to 0^+$ and $s_n \to +\infty$ and letting $\Omega_n = \Omega(\rho_n,s_n)$, by Proposition \ref{prop:qghorse} for any \( \mu_{n} \) supported on \( \Omega_{n} \) we have \( \mu_{n}\in
\mathcal O (3 \rho_{n}, s_{n}) \) and therefore  \( \mu_{n}\to \mu \).

We start with a technical Lemma. 
Let \( m_{1},..., m_{N} \) be the ``return times'' associated to the hyperbolic quasi-generic branches which define the horseshoe \( \Omega^{*}(\rho, s) \). Then we let 
\begin{equation}\label{saturation.time}
T(\rho,s) := \dfrac{\max\{m_1, \cdots , m_N\}\max\{\|\phi_i\|_{\infty}: i = 1, \cdots s\}}{\rho}.
\end{equation}

\begin{lemma}
For all  $x \in \Omega$ we have 
\begin{equation}\label{estimative.0}
\left|\dfrac{1}{L}\sum_{j=0}^{L-1}\phi_i(f^j(x)) - \int\phi_i{d{\mu}}\right| < 2\rho, \quad\forall i \leq s, \ \forall L \geq T(\rho,s).
\end{equation}
\end{lemma}

\begin{proof}

Recall that \( \Omega \) is the saturate of \( \Omega^{*} \) and so there exists some finite iterate of \( x \) which belongs to \( \Omega^{*} \) and so we may  suppose without loss of generality that $x \in \Omega^*$. We now fix \( L\geq T(\rho, s) \) and consider the orbit of \( x, f(x),..., f^{L-1}(x) \) of \( x \) up to time \( L-1 \).  By construction of \( \Omega \), the point \( x \) returns repeatedly to \( \Omega^{*} \) and the number of iterates between two returns depends on  which stable strip \( S_{i} \) that particular iterate of \( x \)  belongs to. We keep track of the combinatorics of these returns by introducing the following notation. Let \( N \) denote the number of branches of the variable time horseshoe \( \Omega^{*} \) (for the purposes of proving our Theorem it would be sufficient to consider the case \( N=2 \) but we consider the general case here for greater generality), then for each \( i=1,..., N \), let \( C_{k} \) denote the number of returns of the point \( x \) to the stable strip \( S_{k} \) before time \( L \). Following each such return the orbit of the point \( x \) belongs to the image of the stable strip \( S_{k} \) for \( m_{k} \) iterates, after which time it returns once again to \( \Omega^{*} \) and falls into another strip \( S_{k'} \). 
We let \( \mathcal L_{\ell}^{(k)} \) denote the set of \( m_{k} \) consecutive iterates following the \( \ell \)'th return of \( x \) to the stable strip \( S_{k} \). Using this notation we can then write 
\[ 
L = L'+ R \quad\text{ where } \quad L'=\sum_{k=1}^NC_km_k \ \text{ and } \ 0 \leq R < \max_k\{m_k\} 
\] 
and therefore 
\begin{equation}\label{2}
 \sum_{j=0}^{L-1}\phi_i(f^j(x)) =   \sum_{k=1}^N\sum_{\ell=0}^{C_k}\sum_{j \in {\mathcal L}^{(k)}_l}\phi_i(f^j(x))+  \sum_{L' \leq j < L}\phi_i(f^j(x)). 
\end{equation}
We are now ready to begin to estimate \eqref{estimative.0}. First of all we write 
\begin{equation}\label{1}
\left|\dfrac{1}{L}\sum_{j=0}^{L-1}\phi_i(f^j(x)) - \int\phi_i{d{\mu}}\right| = \dfrac{1}{L}\left| \sum_{j=0}^{L-1}\phi_i(f^j(x))-L\int\phi_i{d{\mu}}\right|
\end{equation}
The right hand side of \eqref{1} is bounded by 
\[ 
\frac 1L\left| \sum_{k=1}^N\sum_{\ell=0}^{C_k}\sum_{j \in {\mathcal L}^{(k)}_\ell}\phi_i(f^j(x))+  \sum_{L' \leq j < L}\phi_i(f^j(x)) - \left(L'+R\right)\int\phi_i{d{\mu}}\right|
 \] 
which is in turn bounded by 
\begin{equation}\label{bigsum}
\frac 1L\left| \sum_{k=1}^N\sum_{\ell=0}^{C_k}\sum_{j \in {\mathcal L}^{(k)}_\ell}\phi_i(f^j(x))- L' \int\phi_i{d{\mu}}\right| + \frac 1L\left|
\sum_{L' \leq j < L}\phi_i(f^j(x)) - R\int\phi_i{d{\mu}}\right|.
 \end{equation}
To bound the first term of \eqref{bigsum}, we use the definition of \( L' \) and write 
\[ \frac 1L\left| \sum_{k=1}^N\sum_{\ell=0}^{C_k}\sum_{j \in {\mathcal L}^{(k)}_\ell}\phi_i(f^j(x))- L' \int\phi_i{d{\mu}}\right| \leq  
\dfrac{1}{L}\sum\limits_{k=1}^N\sum\limits_{\ell=0}^{C_k}m_k\left|\dfrac{1}{m_k}\sum\limits_{j \in {\mathcal L}^{(k)}_\ell}\phi_i(f^j(x))-\int\phi_i{d{\mu}} \right|
 \] 
By the assumption that all the branches are \( (\rho, s) \) quasi-generic, for every \( k=1,..., N \) and every \( \ell=1,.., C_{k} \) we have 
\begin{equation}\label{0}
\left|\dfrac{1}{m_k}\sum_{j \in {\mathcal L}^{(k)}_\ell}\phi_i(f^j(x)) - \int\phi_{i}{d{\mu}}\right| < \rho, \quad\forall \ i \leq s.
\end{equation}
and so this gives 
\begin{equation}\label{bigsum1} 
\frac 1L\left| \sum_{k=1}^N\sum_{\ell=0}^{C_k}\sum_{j \in {\mathcal L}^{(k)}_\ell}\phi_i(f^j(x))- L' \int\phi_i{d{\mu}}\right| \leq \frac{L'}{L}\rho \leq \rho.
 \end{equation}
Now, to bound the second term of \eqref{bigsum} we use the fact that \( L-L'=R \) and the fact that  \( L \geq T(\rho, s) \) to get 
\begin{equation}\label{bigsum2} 
\frac 1L\left|
\sum_{L' \leq j < L}\phi_i(f^j(x)) - R\int\phi_i{d{\mu}}\right|\leq 
\dfrac{2R}{L}\max\{\|\phi_i\|_{\infty}: i = 1, \cdots s\} \leq \rho.
 \end{equation}
Substituting \eqref{bigsum1} and \eqref{bigsum2} into \eqref{bigsum} and then into 
\eqref{1} completes the proof. 

\end{proof}

\begin{proof}[Proof of Proposition \ref{prop:qghorse}]
Let 
 $\mu_{\Omega}$ be any ergodic probability measure supported on $\Omega$ and $x$ a generic point for $\mu_{\Omega}$. Then, by Birkhoff's ergodic Theorem,  for all sufficiently large \( L \) we have 
\begin{equation}\label{estimative.1}
\left|\dfrac{1}{L}\sum_{j=0}^{L-1}\phi_i(f^j(x)) - \int\phi_i{\mu_{\Omega}}\right| < \rho, \quad\forall \ i \leq s.
\end{equation}
By the triangle inequality we can write 
\[ 
 \left|\int\phi_i{d{\mu}} - \int\phi_i{d{\mu_{\Omega}}}\right| \leq 
 \left|\dfrac{1}{L}\sum_{j=0}^{L-1}\phi_i(f^j(x)) - \int\phi_i{d{\mu}}\right|  + 
 \left|\dfrac{1}{L}\sum_{j=0}^{L-1}\phi_i(f^j(x)) - \int\phi_i{\mu_{\Omega}}\right|
 \] 
 and therefore, from 
\eqref{estimative.0} and \eqref{estimative.1} we get, for sufficiently large \( L \), 
\[
\left|\int\phi_i{d{\mu}} - \int\phi_i{d{\mu_{\Omega}}}\right| < 3\rho\quad \forall \ i \leq s
\]
which implies that \( \mu_{\Omega}\in \mathcal O(3\rho, s) \) as required. 
\end{proof}

\end{document}